\documentclass[11pt,twoside,]{amsart}
\usepackage{amsmath, amsthm, amscd, amsfonts, amssymb, graphicx, color}
\usepackage[bookmarksnumbered, plainpages]{hyperref}
\newtheorem{thm}{Theorem}[section]
\newtheorem{cor}[thm]{Corollary}
\newtheorem{lem}[thm]{Lemma}
\newtheorem{prop}[thm]{Proposition}
\newtheorem{defn}[thm]{Definition}

\numberwithin{equation}{section}
\def\pn{\par\noindent}

\begin{document}
\begin{tabular}{c r}
\vspace{-0.6cm}

\end{tabular}
\hskip 1.5 cm
\begin{tabular}{l}

\end{tabular}
\hskip 2 cm
\begin{tabular}{c c}
\vspace{-0.2cm}
\vspace{-1.2cm}
\end{tabular}
\vspace{1.3 cm}


\title{On central automorphisms of groups and nilpotent rings}
\author{Yassine Guerboussa and Bounabi Daoud}

\maketitle

\begin{abstract}
 Let $G$ be a group.  The  central automorphism group  $Aut_c(G)$ of $G$ is  the centralizer of $Inn(G)$ the subgroup of $Aut(G)$ of inner automorphisms.  There is a one to one map $ \sigma \mapsto  h_\sigma$  from the set $Aut_c(G)$  onto the set  $Hom(G,Z(G))$ of  homomorphisms from $G$ onto its center, with $ h_\sigma(x)=x^{-1} \sigma(x)$.  This map can be used to obtain informations about the size of $Aut_c(G)$, and also about its structure in some special cases.  In this paper we see how to use it to obtain informations about the structure of $Aut_c(G)$ in the general case.  The notion of the adjoint group of a ring is the main tool in our approach.      
\end{abstract}

\vskip 0.2 true cm


\pagestyle{myheadings}
\markboth{\rightline {\scriptsize  Y. Guerboussa and B. Daoud}}
         {\leftline{\scriptsize On central automorphisms of groups and nilpotent rings }}

\bigskip
\bigskip


\section{\bf Introduction}
\vskip 0.4 true cm
     
  It is very difficult to prove general theorems about the automorphisms of finite p-groups, and very little is known about them.  An automorphism of a group $G$ is termed central if it commutes with every inner automorphism, clearly the central automorphisms of $G$ form a normal subgroup $Aut_c(G)$ of $Aut(G)$.  If $G$ is a finite $p$-group, then $Aut_c(G)$ has a great importance in investigating $Aut(G)$, and it has been studied by several authors, see for instance ([2]-[5], and also [9], [10]).\\
It is easy to see that the map, or the Adney-Yen map for convenience, $ \sigma \mapsto  h_\sigma$ determines a one to one map from the set $Aut_c(G)$ onto the set $Hom(G,Z(G))$, where $ h_\sigma(x)=x^{-1} \sigma(x)$.  What are the informations that can be deduced about $Aut_c(G)$ from this relation? this is the main task of this paper. \\
  Let $R$ be a (associative) ring.  Under the circle composition $x\circ y=x+y+xy$, the set of all elements of $R$ forms a monoid with identity element $0 \in R$, this monoid is called the adjoint monoid or semigroup of the ring $R$.  The adjoint group $ R^\circ$ of $R$ is the group of invertible elements in this monoid. 
\\ Let consider the set $Hom(G,Z(G))$ as a ring, the addition is defined in the usual way and we take the composition of maps as a multiplication.  Our main observation is that the Adney-Yen map defines an isomorphism between $Aut_c(G)$ and the adjoint group of the ring $Hom(G,Z(G))$.    
 \\ When the ring $R$ has an identity $1$, the mapping $x \mapsto 1+x$ determines a group isomorphism from $R^\circ$ to the multiplicative group of the ring $R$.  This agrees with the usual case when $G$ is abelian : the central automorphism group coincides with $Aut(G)$ which is the multiplicative group of the ring $End(G)$.     
\\Assume that $G$ is finite.  It was proved in [2] that the Adney-Yen map is a bijection if $G$ does not have a non-trivial abelian direct factor.  In the light of our observation, this is equivalent to saying that $Hom(G,Z(G))$ is a radical ring.  Following Jacobson, a ring $R$ is termed radical if its adjoint semigroup is a group, or equivalently $ R^\circ = R$.  Adjoint groups of  radical rings are interesting objects to study and we may find a considerable number of papers in the subject (see [6] for some references).
\\The above results and some of its consequences are discussed in Section 2 in a more general context.  And since we are mainly interested to finite p-groups, the remaining sections are devoted to their central automorphisms, in Section 3 we introduce the notion of a $p$-nil ring in order to studying the structure of $Aut_c(G)$ when $G$ is a finite $p$-group with $Z(G) \leq \Phi (G)$.  The results of this section are applied in Section 4 to the longstanding problem of whether every non-abelian finite $p$-group has a non-inner automorphism of order $p$ (see [1]), we give a necessary and a sufficient condition for a finite $p$-groups to have a non-inner central automorphism of order $p>2$.      

Throughout, the unexplained notation is standard in the literature.  We denote by $Hom(G,N)$ the group of homomorphisms from $G$ to an abelian group $N$.  We denote by $d(G)$ the minimal number of generators of $G$, and the rank $r(G)$ of $G$ is defined to be $sup\{d(H), H\leq G\}$.  The exponent of $G$ is denoted by $exp(G)$ and $\mathbb{Z}_n$ denotes the ring of integers modulo $n$.  

\begin{lem}
If $M$ and $N$ are finite abelian $p$-groups, then the rank and the exponent of the abelian group $Hom(M,N)$ are equal respectively to $r(M). r(N)$ and $min\{exp(M),exp(N)\}$.
\end{lem}
\begin{proof}
This follows immediately from the properties $$Hom(\prod_{i}M_i,\prod_{j}N_j) \cong \prod_{i,j}Hom(M_i,N_j)$$
 where $M_i$ and $N_j$ are abelian groups, and $$Hom(\mathbb{Z}_{p^n},\mathbb{Z}_{p^m}) \cong \mathbb{Z}_{p^{min\{n,m\}}}.$$
\end{proof}

Given an associtive ring $R$, we denote by $R^+$  the additive group of $R$.  The $n{th}$ power $R^n$ of $R$ is the additive group generated by all the products of $n$ elements of $R$.  We say that $R$ is nilpotent if  $R^{n+1}=0$ for some non-negative integer $n$, the least integer $n$ satisfying $ R^{n+1}=0$ is called the class of nilpotency of the ring $R$.  Note that every nilpotent ring $R$ is radical since for every $x \in R$ we have$$ x\circ \sum_{i}(-1)^ix^i=(\sum_{i}(-1)^ix^i) \circ x =0.$$
\\The Jacobson radical of the ring $R$ is the largest ideal of $R$ contained in the adjoint group $R^\circ$.  This implies that $R$ is radical if and only if it coincides with its Jacobson radical.  By a classical result the Jacobson radical of an artinian ring is nilpotent, so every artinian (in particular finite) radical ring is nilpotent.\\
The following lemma is standard in the literature (see [8], Section I.6).
\begin{lem}
The adjoint group of a nilpotent ring $R$ is nilpotent of class at most equals to the nilpotency class of $R$.
\end{lem}  
\begin{proof}
The series of ideals $$ R \supset R^2  \supset ...\supset R^{n+1}=0$$
induces a central series in the adjoint group of the ring $R$.
 \end{proof}
 The following lemma is a variant of theorem B in [6], it gives a bound for the rank of the adjoint group of a finite (periodic in general) radical ring $R$ in term of the rank of its additive group.
\begin{lem}
Let $R$ be a finite radical ring.  Then $ r(R^\circ) \leq 3 r(R^+)$, and if the order of $R$ is odd then $ r(R^\circ) \leq 2 r(R^+)$.
\end{lem}

\section{\bf {\bf \em{\bf  Central automorphisms and radical rings}}}
\vskip 0.4 true cm

We begin with the following general remark.  Every abelian normal subgroup $A$ of a group $G$ can be viewed as a $G$-module via conjugation $a^x=x^{-1}ax$, with $x \in G \mbox{ and }a \in A$.  A derivation of $G$ into $A$ is a mapping $\delta:G\rightarrow A$ such that $\delta(xy)=\delta(x)^x \delta(y)$.  The set $Der(G,A)$ of these derivations is a ring under the addition $ \delta_1+\delta_2(x)= \delta_1(x)\delta_2(x)$ and the multiplication  $ \delta_1 \delta_2(x)= \delta_2(\delta_1(x))$, with  $ \delta_1,\delta_2 \in Der(G,A) \mbox { and }x \in G$.  Let denote by  $End_A(G)$ the set of  endomorphisms $u$ of $G$ having the property $x^{-1}u(x) \in A$, for all $x \in G$.  We check easily that  $End_A(G)$ is a submonoid of $End(G)$ and every endomorphism $u \in  End_A(G)$ defines a derivation $\delta_u(x)=x^{-1} u(x)$ of $G$ into $A$.  Note also that to each derivation $\delta \in Der(G,A)  $ we can associate an endomorphism $u \in End_A(G) $ with $u(x)=x\delta(x)$.
\begin{lem}
Under the above notation, the mapping $u\mapsto \delta_u  $ is an isomorphism between the monoid $End_A(G)$ and the adjoint monoid of the ring $Der(G,A)$.  In particular it induces an isomorphism between the corresponding groups of invertible elements.
\end{lem}
\begin{proof}
Straightforward verification.
\end{proof}
Since the center $Z(G)$ is a trivial $G$-module, we have $Der(G,Z(G))=Hom(G,Z(G))$.  So for $A=Z(G)$ the mapping defined above reduces to the Adney-Yen map.  It follows that
\begin{prop}
The Adney-Yen map determines an isomorphism between the central automorphism group $Aut_c(G)$ and the adjoint group of the ring $Hom(G,Z(G))$.  
\end{prop}

Assume that $G$ is finite.  In [2] Adney and Yen have proved that every endomorphism in $ End_{Z(G)}(G)$  is an automorphism if and only if $G$ is purely non-abelian, that is $G$ does not have a non-trivial abelian direct factor.  The above observation  allows us to set this result under the form
\begin{thm}{ (Adney-Yen)}
Let $G$ be a finite group.  Then the ring  $Hom(G,Z(G))$ is radical if and only if $G$ is purely non-abelian. 
\end{thm}
The above theorem can be generalized to arbitrary finite rings as follows.
\begin{thm}
Let $R$  be a finite ring.  Then $R$  is radical if and only if $0$ is the only idempotent in $R$.
\end{thm}

Let be $R=Hom(G,Z(G))$.  We have $R$ is non-radical if and only if  there exists a non-zero idempotent homomorphism $e: G\rightarrow Z(G)$, and clearly this is equivalent to the existence of a non-trivial abelian direct factor of $G$.

The proof of Theorem 2.4 is based on the following result.
\begin{lem}
Let $x$ be an element of a semigroup $S$ such that $x^n = x^m$ for some positive integers $n\neq m$.  Then the set $\{ x^k \in S, k > 0 \}$ contains an idempotent.
\end{lem}
\begin{proof}
For every $n>0$, let $[n] = \{ k>0, x^k = x^n\}$. 
\\ Assume that $n < min [2n], \mbox{ for all } n>0$.  There exist by assumption $n<m$ such that $x^n = x^m$, so the class $[n]$ is unbounded since $n+k(m-n) \in [n]\mbox{, for all } k>0$.  On the other hand if $ l \in [n]$, then $ 2n\in [2l] $, and so $ l < 2n $, a contradiction.
\\ Hence, there exists $n$ such that $ n_0 = min [2n] \leq n$.  If $n_0 = n$, then $x^n$ is an idempotent element of $S$.  And if $n_0<n$, then $x^{2n-n_0}$ is an idempotent, since
$$(x^{2n-n_0})^2 = x^{4n-2n_0} = x^{2n} x^{2n-2n_0} = x^{n_0} x^{2n-2n_0} =  x^{2n-n_0}.$$
The result follows.
\end{proof}

\begin{proof}[Proof of Theorem 2.4]
Suppose that $R$ is not radical.  Since $R^\circ$ contains every nilpotent element, then $R$ contains a non-nilpotent element $x$. 
And since $R$ is finite, the set of all the powers of $x$ can not be infinite.  Hence there exist $n\neq m$ such that $x^n = x^m$.  The existence of a non-zero idempotent element follows now from Lemma 2.5.
\\Conversely, if $x \neq 0$ is an idempotent of $R$, then $-x\notin R^\circ $.  Otherwise there exists an element $y\in R$ such that $ -x+y-xy = 0$, if we multiply this equation by $x$ on the left we obtain $ -x =0$, which is not the case.  Hence $R^\circ \neq R$, and so $R$ is not radical.  The result follows.  
\end{proof}
 
As an immediate consequence of Theorem 2.3, we have 
\begin{cor}
If $G$ is a purely non-abelian finite group, then the ring $Hom(G,Z(G))$ is nilpotent. In particular, every homomorphism  $h:G \rightarrow Z(G)$ is nilpotent.
\end{cor}
The following corollary is well-known in the litterature  (see [9]).
\begin{cor} 
The central automorphism group of a purely non-abelian finite group is nilpotent.
\end{cor}
We can also bound the rank of $Aut_c(G)$ using Lemma 1.3.  

\begin{cor}
Let $G$ be a purely non-Abelian finite group.  Then $ r(Aut_c(G))\leq 3  r(R^+)$, where $R$ denotes the ring $Hom(G,Z(G))$.  The bound $3$ can be replaced by $2$ if the order of $Z(G)$ is odd. 
In partucular if $G$ is a $p$-group then, $ r(Aut_c(G))\leq 2  d(G)  d(Z(G))\mbox{ for } p>2\mbox{,  and }  r(Aut_c(G))\leq 3  d(G)  d(Z(G)\mbox{ for } p =2$. 
\end{cor}
\begin{proof}
The first part follows from Lemma 1.3.  For the second observe that every homomorphism  $h:G \rightarrow Z(G)$ can be factorized on $G\prime $, this induces an isomorphism between the two groups $Hom(G,Z(G))$ and $Hom(G/G\prime,Z(G))$.  The result follows now from Lemma 1.1. 
\end{proof}
\section{\bf {\bf \em{\bf   Adjoint groups  of $p$-nil rings}}}
\vskip 0.4 true cm

In this section we investigate more closely the structure of  $Aut_c(G)$ when $G$ is a finite  $p$-group with $Z(G) \leq \Phi (G)$.  This situation motivates the introduction of the following notions.  
\begin{defn}
Let $p$ be a prime number and $R$ be a ring.  We say that $R$ is left (right, resp) $p$-nil if every element $x$ of order $p$ in $R^+$ is a left (right, resp) annihilator of $R$, that is $ p x =0$ implies $xy=0$ $ (yx=0\mbox{, resp})\mbox {, for all }y\in R$. The ring $R$ is said to be $p$-nil if it is left and right $p$-nil. 
\end{defn}

For instance, the subring $S=pR$ of any ring $R$ is $p$-nil.  Also we check easily that the left and the right annihilators of  $\Omega_1(R^+)$ are respectively right and left $p$-nil.
\\ The following theorems shed some lights on the structure of the adjoint groups of these rings.

\begin{thm}
Let $R$ be a ring with an additive group of finite exponent $p^m$.  If $R$ is  left or right $p$-nil, then $R$ is nilpotent of class at most $m$.  In particular the adjoint group $R^\circ$ is nilpotent of class at most $m$.
\end{thm}
\begin{proof}
Assume that $R$ is left $p$-nil.  We proceed by induction on $n$ to prove that $p^{m-n+1}R^n=0$.  This is obvious for $n=1$.  Now if $x \in R^n$, then by induction $p^{m-n+1}x=0$.  It follows that $p^{m-n}x$ has order $1$ or $p$, therefore $ (p^{m-n}x)y=p^{m-n}(xy)=0$, for all $y\in R$.  This shows that $p^{m-n}R^{n+1}=0$.  Now, for $n=m+1$ we have $R^{m+1}=0$, this prove that $R$ is nilpotent of class at most $m$.  The result follows  for $R$ right $p$-nil by a similar argument.  The second assertion follows from Lemma 1.2.  
\end{proof}
\begin{lem}
If $R$ is a left (right, resp) $p$-nil ring, then the factor ring $R/\Omega_n(R)$ is left (right, resp) $p$-nil for all $n \geq 1$, where $\Omega_n(R)$ denotes the ideal $\{x \in R, p^nx=0 \}$.
\end{lem}
\begin{proof}
Assume that $R$ is left $p$-nil, and  let be $\overline x \in R/\Omega_n(R) $ such that $ p\overline x = \overline 0$.  Then $ px \in \Omega_n(R)$, so $p^nx \in \Omega_1(R)$, and by assumption $ (p^nx)y=p^n(xy)=0 \mbox{, for all } y \in R$.  This shows that $ xy \in \Omega_n(R)\mbox{, for all } y \in R$, that is $\overline x$ is a left annihilator of $R/\Omega_n(R)$.  The result follows for $R$ right $p$-nil by a similar argument.  
\end{proof}

\begin{thm}
Let $R$ be a $p$-ring, $p$ odd.  If $R$ is left or right $p$-nil, then $\Omega_{\{n\}}(R^\circ)=\Omega_n(R)$, for every $n \geq 1$.  In particular we have $\Omega_n(R^\circ)=\Omega_{\{n\}}(R^\circ)$.  
\end{thm}
\begin{proof}
We denote by $ x^{(k)}$ the $k{th}$ power of $x$ in the adjoint group of $R$.
\\For $n=1$ we have, if $px= 0$ then $x^i =0 \mbox{ for } i\geq 2$.  Hence $$ x^{(p)}=\sum_{i\geq 1}\binom{p}{i} x^i =px=0,$$
 and so $x\in \Omega_{\{1\}}(R^\circ)$. Conversely, if $ x^{(p)}=0$ then $$ px= -\sum_{i\geq 2}\binom{p}{i} x^i.$$
 Let $p^m$ be the additive order of $x$.  If $m\geq 2$, then $ p^{m-1}x$ has order $p$, hence $ p^{m-1}x^2=0$, and similarly we obtain $p^{m-2}x^i=0\mbox{, for } i\geq 3$.  Now if we multiply the above equation by $p^{m-2}$ we obtain $$ p^{m-1}x= -\sum_{i\geq 2}\binom{p}{i} p^{m-2} x^i=0$$ 
This contradicts the definition of the order of $x$.  Therefore $m\leq 1$, and so $x \in \Omega_{1}(R)$.
\\ Now we proceed by induction on $n$.  If $x \in \Omega_n(R)$, then $ px \in \Omega_{n-1}(R)$.  This implies that $ x+\Omega_{n-1}(R) \in \Omega_1( R/\Omega_{n-1}(R))$.  Lemma 3.3  and the first step imply that $x+\Omega_{n-1}(R) \in \Omega_{\{1\}}( (R/\Omega_{n-1}(R))^\circ)$.  Hence $ x^{(p)} \in \Omega_{n-1}(R) $, and by induction $ x^{(p)} \in \Omega_{\{n-1\}}(R^\circ) $.  Thus  $ x \in \Omega_{\{n\}}(R^\circ) $.  It follows that $\Omega_n(R) \subset \Omega_{\{n\}}(R^\circ)$.  The inverse inclusion follows similarly.  
\\Finally, the equality  $\Omega_n(R^\circ)=\Omega_{\{n\}}(R^\circ)$ follows from the fact that $(\Omega_n(R))^\circ$ is a subgroup of $R^\circ$ and $\Omega_n(R^\circ)$ is generated by $\Omega_{\{n\}}(R^\circ)$.                                      
\end{proof}
\begin{cor}
Let $R$ be a $p$-ring, $p$ odd.  If $R$ is $p$-nil, then $\Omega_{1}(R^\circ) \leq Z(R^\circ)$, in other word $R^\circ$ is $p$-central.
\end{cor}
\begin{proof}
Every element $x$ of  $\Omega_{1}(R^\circ)\mbox{ lies }  \Omega_{1}(R)$  by the above theorem.  Hence $x$ is an annihilator of $R$, and so it lies in the center of $R^\circ$.     
\end{proof}
Note that this can be used to prove Lemma 1.3 among the same lines of Dickenschied proof ([6]), only we use the fact that the group $(pR)^{\circ}$ is p-central instead of being powerful (a finite $p$-group $G$ is powerful if $G/G^p$ ($G/G^4$, for $p=2$) is abelian), and the fact that the rank of a p-central finite p-group $G$ is bounded by $d(Z(G))$ by a result of Thompson (see [7, III, Hilfssatz 12.2]).  It seems that this alternative proof is simpler, since it is easier to prove that $(pR)^{\circ}$ is p-central than proving that is powerful, but unfortunately this proof does not deal with the prime $p=2$.

In connection with central automorphisms we have
\begin{prop}
If $G$ is a finite $p$-group such that $Z(G) \leq \Phi (G)$, then the ring $Hom(G,Z(G))$ is right $p$-nil.
\end{prop}
\begin{proof}
Let be $k,h \in Hom(G,Z(G))$ such that $ph=0$.  Then $h :G \rightarrow \Omega_1(Z(G))$.  Since the image of $h$ is an elementary abelian $p$-group, its kernel contains the frattini subgroup, and since $Z(G) \leq \Phi (G)$ we have $kh(x)=h(k(x))=1 \mbox{, for all } x \in G$.  It follows that $h$ is a right annihilator of the ring $Hom(G,Z(G))$.  
\end{proof}
The above proposition leads to a new proof of Theorem 4.8  in [9].
\begin{cor}
If $G$ is a finite $p$-group such that $Z(G) \leq \Phi (G)$, then $Aut_c(G)$ is nilpotent of class at most $min\{r,s\}$, where $exp(G/G\prime)=p^r$ and $exp(Z(G))=p^s$. 
\end{cor}
\begin{proof}
By Theorem 3.2 the nilpotency class of $Aut_c(G)$ does not exceed $m$, where $p^m$ is the exponent of $Hom(G,Z(G)) \cong Hom(G/G\prime,Z(G))$ which is equal to $p^{min\{r,s\}}$ by Lemma 1.1.
\end{proof}

\begin{thm}
If $G$ is a finite $p$-group with $p$ odd, such that $Z(G) \leq \Phi (G)$, then $$ \Omega_{n}(Aut_c(G))= \Omega_{\{n\}}(Aut_c(G))= Aut_{Z_n}(G) $$  where $Z_n$ denotes the subgroup $ \Omega_{n}(Z(G))$. 
\end{thm}
\begin{proof}
This is an immediate consequence of Theorem 3.4 and Proposition 3.6.
\end{proof}
\section{\bf {\bf \em{\bf   Non-inner central automrphisms of order $p$.}}}
\vskip 0.4 true cm
A longstanding conjecture asserts that every non-abelian finite $p$-group has a non-inner automorphism of order $p$.  More informations about this conjecture can be found for instance in [1]. 
\\First, note that we can reduce it to indecomposable $p$-groups.  
\begin{prop}
Let $G$ be a non-abelian finite $p$-group.  If $G$ is decomposable then $G$ has a non-inner central automorphism of order $p$. 
\end{prop}
\begin{proof}
Assume that $G$ is a direct product of $G_1$ and $G_2$, where $G_1$, $G_2$  are non-trivial normal subgroups of $G$.  Let $M$ be a maximal subgroup of $G_1$ and  $g \in{G_1-M}$, clearly every element of $G$ can be written in the form $xg^i$, where $x \in M G_2$.  If $z$  is a central element of order $p$ in $G_2$, then the mapping $ xg^i \mapsto xg^iz^i$ is a central automorphism of $G$ of order $p$ which is not inner since it maps $g \in G_1$ to $gz \notin G_1$. 
\end{proof}
For $p$ odd, the results of the previous section allows us to caracterize the $p$-groups in which every central automorphism of order $p$ is inner.  Let denote $d=d(G)$, $d_1=d(Z(G))$ and $d_2=d(Z(Inn(G))$.    
\begin{thm}
Let $G$ be a finite non-abelian $p$-group, $p$ odd.  In order for $G$ to have a non-inner central automorphism of order $p$ it is necessary and sufficient that $d_2 \neq d\cdot d_1$.
\end{thm}
For instance, the $p$-groups of maximal class satisfy this condition, as well as the class of non-abelian finite $p$-central $p$-groups, this follows easily from [7, III, Hilfssatz 12.2]. 
We need the following two lemmas to prove Theorem 4.2.
\begin{lem}
If $G$ is a purely non-abelian finite $p$-group, then  $\Omega_1(Z(G))\leq  \Phi (G)$.  In particular if $exp(Z(G))=p$ then $Z(G)\leq  \Phi (G)$.
\end{lem}
\begin{proof}
Let $z \in  \Omega_1(Z(G))$.  If there exists a maximal subgroup $M$ such that $z \not \in M$, then $G \cong <z> \times M$. Thus $G$  is not purely non-abelian. 
\\This is another proof based on the nilpotency of the ring $Hom(G,Z(G))$.
Let be $z \in \Omega_1(Z(G))$.  To each homomorphism $r : G \rightarrow \mathbb{Z}_p$  we can associate an endomorphism $h \in Hom(G,Z(G))$ by setting $h(x)=r(x)z  \mbox{, for all } x \in G$.  This implies that $h^n(z)=r(z)^nz$.  By corollary 2.6, $h$ is nilpotent, so there exists an integer $n$ such that $r(z)^n=0 $. Therefore $r(z)=0$, since $\mathbb{Z}_p$ is a field.  This shows that $z$ lies in the intersection of the set of all kernels of homomorphisms from $G$ to $\mathbb{Z}_p$. Since every maximal subgroup of $G$ occurs as a kernel of some homomorphism  $r:G \rightarrow \mathbb{Z}_p$. It follows that $z \in  \Phi (G)$.
\end{proof}

\begin{lem}
Let $G$ be a finite $p$-group.  Then  every inner automorphism which is central of order $p$ is induced by some non-trivial homomorphism $h :G \rightarrow \Omega_1(Z(G))$.  Moreover, if $G$ is purely non-abelian then for every non-trivial homomorphism $h :G \rightarrow \Omega_1(Z(G))$, the order of the central automorphism $\sigma = 1_G+h$ induced by $h$ is equal to $p$. 
\end{lem}
 \begin{proof}
Let be $ \tau$ an inner central automorphism of order $p$.  We can write $\tau=1_G+h$, for some  $h \in Hom(G,Z(G))$, and $1_G$ denotes the identity map of $G$. We have $ h=h\tau=h+h^2$, and so  $h^2=0$.  This implies that $1_G=\tau^p=1_G+ph$, and so $ph=0$.  Therefore  $h :G \rightarrow \Omega_1(Z(G))$.
 \\Assume that $G$ is purely non-abelian.  Since the kernel of every homomorphism  $h :G \rightarrow \Omega_1(Z(G))$ contains $\Phi(G)$,  Lemma 4.3 implies that $h^2=0$.  Therefore, if $\sigma = 1_G+h$ then $\sigma^p = 1_G+\sum_{i=1}^{p}\binom{p}{i} h^i=1_G+ph=1_G$.  The result follows. 
\end{proof}

\begin{proof}[Proof of Theorem 4.2]
Suppose that $G$ has a non-inner central automorphism $\sigma = 1_G+h_{\sigma}$ of order $p$.  Let $I$ be the image of $ \Omega_1(Z(innG))$ by the Adney-Yen map.  By Lemma 4.4 $I$ is a subspace of the $\mathbb{Z}_p$-vector space $Hom(G, \Omega_1(Z(G)))$ of dimension $d_2$.  If  $d_2 = d \cdot d_1$, then $I = Hom(G, \Omega_1(Z(G)))$.   If $Z(G)\nleq  \Phi (G)$ then we can find an element $g \in Z(G)-M$ for some maximal subgroup $M$ of $G$.  Consider a non-trivial element $z \in \Omega_1(Z(G))\cap M$  and let $h(x)=z^{r(x)}$, where $r:G \rightarrow \mathbb{Z}_p$ is the homomorphism defined by $r(mg^i)=i \mod p$, $m \in M$.  Clearly, $ h \in I$ and $1_G+h$ is not inner, since it maps $g$ to $gz$, a contradiction.  It follows that $Z(G)\leq  \Phi (G)$.  Theorem 3.8 implies that $ph_{\sigma}=0$, that is $h_{\sigma} \in I$.  It follows that $\sigma = 1_G+h_{\sigma}$ is inner, a contradiction.  Therefore $d_2 \neq d \cdot d_1$. 
\\Conversely,  by Proposition 4.1 we may suppose that $G$ is purely non-abelian.  If  $d_2 \neq d\cdot d_1$, then $I$ is a proper subspace of  $Hom(G, \Omega_1(Z(G)))$.  Hence there exists  $h :G \rightarrow \Omega_1(Z(G))$ such that the automorphism $\sigma = 1_G+h$ is not inner.   It follows from Lemma 4.4 that $\sigma $ has order $p$.
\end{proof}
 Let $G$ be a finite non-abelian $p$-group of order $p^n$ and class $c$.  Under the above notation, does the equality $d_2 = d \cdot d_1$ imply that $G$ has a cyclic center?.  
\\ Assume that $G$ is a counter example to this question, by a formula of Abdollahi [1, Theorem 2.5] we have $ d_1\cdot (d+1) \leq r+1$, where $r=n-c$ is the coclass of $G$.
The class of $G$ must be $\geq 3$, otherwise we would have $d_2=d(G/Z(G)) \leq d$, so $d_1=1$ which is not the case.  On the other hand $d \geq 2$, hence $ 3d_1 \leq r+1$, so we must have $r+1 \geq 6$, thus $n \geq 5+c \geq 8$.
\\This shows that if a counter example to the above question exists then it has at least coclass $5$ and order $ p^8$.  It is 
well-known that in a powerful $p$-group $G$, every subgroup can be generated by $d(G)$ elements, so a counter example to our question can not be a powerful $p$-group.
\subsection*{Acknowledgment}
The first author is grateful to Miloud Reguiat for his encouragement, and his comments about early drafts of this paper.

\vskip 0.4 true cm



\bigskip
\bigskip


{\footnotesize \pn{\bf Yassine Guerboussa}\; \\ {Department of
Mathematics}, {University Kasdi Merbah Ouargla,} {Ouargla, Algeria}\\
{\tt Email: yassine\_guer@hotmail.fr}\\
{\footnotesize \pn{\bf Bounabi Daoud}\; \\ {Department of
Mathematics}, {University of Setif 
,} {Setif, Algeria}\\
{\tt Email: boun\_daoud\@yahoo.com}\\


\begin{thebibliography}{20}
\bibitem{Abd}
A. Abdollahi,  Powerful $p$-groups have non-inner automorphisms of order $p$ and some cohomology, {\em  J. Algebra.} {\bf 323}
(2010), 779-789.


\bibitem{Ad}
\bibitem{Abd}
A. Abdollahi,  Powerful $p$-groups have non-inner automorphisms of order $p$ and some cohomology, {\em  J. Algebra.} {\bf 323}
(2010), 779-789.


\bibitem{Ad}
J.E. Adney and T. Yen, Automorphisms of a $p$-group, {\em Illinois J. Math.} {\bf 9} (1965), 137-143.
\bibitem{Att}
M.S. Attar, Finite $p$-groups in which each central automorphism fixes centre elementwise, 
 {\em Comm. Algebra.} {\bf 40} (2012), 1096-1102.
 
\bibitem{C}
M.J. Curran, Finite groups with central automorphism group of minimal order, 
 {\em Math. Proc. Royal Irish Acad.} {\bf 104 A(2)} (2004), 223-229.
 
\bibitem{C}
M.J. Curran and D.J. McCaughan, Central automorphisms that are almost inner, 
 {\em Comm. Algebra.} {\bf 29 (5)} (2001), 2081-2087.
\bibitem{D}
O. Dickenschied, On the adjoint group of some radical rings, 
 {\em Glasgow Math. J.} {\bf 39} (1997), 35-41.
\bibitem{H}
B. Huppert. Endliche Gruppen. I. {\em Die Grundlehren der Mathematischen Wissenschaften},
Band 134. Springer-Verlag, Berlin, 1967.
\bibitem{KP}
R.L. Kruse and D.T. Price,  {\em Nilpotent rings},
Gordon and Breach, New York (2010).
  
\bibitem{J}
M.H. Jafari and A.R. Jamali, On the nilpotency and solubility of the central automorphism group of finite group, 
 {\em  Algebra Coll.} {\bf 15}:3 (2006), 485-492.

\bibitem{J}
M.K. Yadav, On central automorphisms fixing the center element-wise, 
 {\em Comm. Algebra.} {\bf 37} (2009), 4325-4331.


\end{thebibliography}
\end{document}